\documentclass[oneside,english]{amsart}
\usepackage[T1]{fontenc}
\usepackage[latin9]{inputenc}
\usepackage{amsthm}
\usepackage{amstext}
\usepackage{amssymb}
\usepackage{esint}
\usepackage{amscd}
\usepackage{color}
\usepackage{graphics}
\usepackage{graphicx}
\makeatletter
\numberwithin{equation}{section}
\numberwithin{figure}{section}
\theoremstyle{plain}
\newtheorem{thm}{\protect\theoremname}[section]

\theoremstyle{plain}
\theoremstyle{definition}

\theoremstyle{plain}

\newtheorem{cor}[thm]{\protect\corollaryname}
\theoremstyle{plain}
\newtheorem{rem}[thm]{\protect\remarkname}
\theoremstyle{plain}
\makeatother

\usepackage{babel}
\providecommand{\definitionname}{Definition}
\providecommand{\lemmaname}{Lemma}
\providecommand{\theoremname}{Theorem}
\providecommand{\corollaryname}{Corollary}
\providecommand{\remarkname}{Remark}
\providecommand{\propositionname}{Proposition}

\DeclareMathOperator{\loc}{loc}

\DeclareMathOperator{\ess}{ess}

\newcommand{\I}{\mathrm{I}}

\begin{document}

\title[Sobolev $(p,q)$-extension operators and Neumann eigenvalues]{Sobolev $(p,q)$-extension operators and  Neumann eigenvalues}

\author{Vladimir Gol'dshtein and Alexander Ukhlov}

\begin{abstract}
In this article, we consider $(p,q)$-extension operators, $1 < q \le p < \infty$, on Sobolev spaces. Based on composition operators on Sobolev spaces, we construct the extension operators in outward cuspidal domains with estimates of their norms. Using these $(p,q)$-extension operators, we prove estimates for the non-linear Neumann eigenvalues of the $p$-Laplace operator in outward cuspidal domains.
\end{abstract}

\maketitle

\section{Introduction}

Let $L^1_p(\Omega)$, where $\Omega$ is a domain in $\mathbb{R}^n$, $n \geq 2$, be the seminormed Sobolev space. 
Recall that $\Omega$ is called a $(p,q)$\textit{-extension domain}, $1 \leq q \leq p \leq \infty$, if there exists a bounded operator
\[
E: L^1_p(\Omega) \to L^1_q(\mathbb{R}^n),
\]
such that $E(f)\big|_{\Omega} = f$ for every function $f \in L^1_p(\Omega)$. 
This operator is referred to as a bounded $(p,q)$\textit{-extension operator}.

We define the operator bound as a measure of its boundedness:
\[
\|E\| = \sup_{f \in L^1_p(\Omega) \setminus \{0\}} \frac{\|E(f)\|_{L^1_q(\mathbb{R}^n)}}{\|f\|_{L^1_p(\Omega)}}.
\]
If $E$ is linear, then $\|E\|$ is the usual operator norm.

Two problems are addressed in the present article. The first problem concerns the explicit construction of bounded $(p,q)$\nobreakdash-extension operators with quantitative estimates of their norms in outward cuspidal domains. 
The study of sharp constants in Sobolev-type inequalities goes back to the classical monograph of G.~P\'olya and G.~Szeg\H{o}~\cite{PS51}, and remains one of the central themes in geometric analysis. 
In particular, obtaining explicit bounds for the norms of extension operators is a subtle and difficult problem; see, for instance, \cite{Mik79}.

The second problem concerns spectral estimates for Neumann eigenvalues of non-linear elliptic operators in non-Lipschitz domains obtained by means of these extension norm bounds. 
The main results are summarized below; detailed statements and proofs are given in the subsequent sections.

The first problem goes back to the classical works of A.~P.~Calderon~\cite{C61} and E.~M.~Stein~\cite{S70}, who proved that bounded Lipschitz domains are $(p,p)$\nobreakdash-extension domains for all $1 \leq p \leq \infty$. 
This result was later generalized by P.~Jones~\cite{J81} to the class of so\mbox{-}called uniform domains.

However, obtaining estimates for the norms of extension operators is a more complicated problem. 
Recently, following the approach of~\cite{S70}, estimates for the norms of extension operators in Lipschitz domains were obtained in~\cite{T15}. In the present article we obtain bounds for the norms of extension operators in outward cuspidal domains, which poses a significant technical challenge. These results are crucial for proving estimates of Neumann eigenvalues of elliptic operators~\cite{GPU20}.

The first necessary and sufficient conditions for the boundedness of the extension operator
\[
E: L^1_2(\Omega) \to L^1_2(\mathbb{R}^2), \quad \Omega \subset \mathbb{R}^2,
\]
were established in~\cite{VGL79}. It was shown that a simply connected domain
\(\Omega \subset \mathbb{R}^2\) is a $(2,2)$\nobreakdash-extension domain \emph{if and only if}
\(\Omega\) is an Ahlfors domain~\cite{Ahl66} (that is, a quasidisk).
Necessary and sufficient conditions for the existence of $(p,p)$-extension operators, $p>2$,
in planar simply connected domains were obtained in~\cite{Sh10,SZ16}
in terms of the subhyperbolic metric.

However, the characterization of $(p,q)$\nobreakdash-extension domains remains an open and complicated problem, despite intensive research in geometric analysis and PDEs. We note the works~\cite{AO12,GS82,HKT08,KRZ25,KUZ22,KZ22,MP86,MP87,RZ24,U20}, which are devoted to the $(p,q)$\nobreakdash-extension problem.

In the present work, based on the geometric theory of composition operators on Sobolev spaces~\cite{VU04, VU05}, we construct $(p,q)$-extension operators in outward cuspidal domains and obtain explicit estimates of their operator norms. Remark that the geometric theory of composition operators on Sobolev spaces goes back to~\cite{U93} and forms a significant part of the modern geometric analysis of Sobolev spaces; see, for example,~\cite{HK12,KKSS14,Mar90,MU24,PV25}.

The first main result of the present article gives the following sharp $(p,q)$-extension theorem together with an explicit norm estimate:

\medskip
\noindent
\textit{Let $\widetilde{\Omega}_{\widetilde{\gamma}} \subset \mathbb{B}^n$, where $\mathbb{B}^n \subset \mathbb{R}^n$ is the unit ball, be an outward $\widetilde{\gamma}$-cuspidal domain with $1<\widetilde{\gamma}<\infty$. Set $\gamma = \widetilde{\gamma} - 1$. Then, for any $\frac{n+(n-1)\gamma}{n} < p < \infty$, the extension operator}
\[
E : L^1_p(\widetilde{\Omega}_{\widetilde{\gamma}}) \to L^1_{q}(\mathbb{B}^n)
\]
\textit{is bounded for all $1 \le q < \frac{n p}{n + (n-1)\gamma}$, and satisfies}
\[
\|E\| \le \left( |\mathbb{B}^n|^{1-\alpha} + |\mathbb{S}^{n-2}|^{\,1-\alpha} \,C_{\gamma}^{\alpha} \left[ \frac{1-\alpha}{\,n(1-\alpha) - \alpha (n-1)\gamma\,} \right]^{1-\alpha} \right)^{\frac{1}{q}},
\]
\textit{with $\frac{1}{p} \le \alpha < \frac{n}{n + (n-1)\gamma}$ arbitrary and}
$$
C_{\gamma} = \pi \left(1 + (\gamma+1)^2 \frac{\pi^4}{(\pi-1)^4} + \frac{1}{(\pi-1)^2}\right)^{\frac{p}{2}}.
$$

\medskip

Based on the obtained extension norm estimates, we turn to spectral problems for the $p$-Laplace operator. 
Estimates for the first non-trivial Neumann eigenvalue of the $p$-Laplace operator in non-convex domains represent a long-standing and difficult problem~\cite{ENT,PS51}.
 In the classical linear case \(p=2\), Payne and Weinberger~\cite{PW} proved that for any convex domain \(\Omega \subset \mathbb{R}^n\), \(n \ge 2\),
\[
\mu(\Omega) \,\geq\, \frac{\pi^2}{d(\Omega)^2},
\]
where \(d(\Omega)\) denotes the Euclidean diameter of~\(\Omega\).

In the case of general non-convex domains, such an estimate in terms of \(d(\Omega)\) is no longer valid. A simple counterexample is provided by a domain composed of two identical squares connected by a narrow corridor~\cite{BCDL16}, for which \(\mu(\Omega)\) can be made arbitrarily small while the diameter \(d(\Omega)\) remains fixed.

These norm bounds for extension operators serve as the principal tool in deriving sharp lower estimates for the first non-trivial Neumann eigenvalue in the following nonlinear eigenvalue problem:
\begin{equation*}
-\Delta_p u:=-\text{div}(|\nabla u|^{p-2}\nabla u)=\mu\|u\|_{L_r(\widetilde{\Omega}_{\widetilde{\gamma}})}^{p-r}|u|^{r-2}u\text{ in }\widetilde{\Omega}_{\widetilde{\gamma}},\quad\frac{\partial u}{\partial\nu}=0\text{ on }\partial \widetilde{\Omega}_{\widetilde{\gamma}}.
\end{equation*}

The corresponding second main result of the present article states:

\medskip
\noindent
\textit{Let $\widetilde{\Omega}_{\widetilde{\gamma}} \subset \mathbb{B}^n$ be an outward $\widetilde{\gamma}$-cuspidal domain, $1<\widetilde{\gamma}<\infty$. Set $\gamma = \widetilde{\gamma} - 1$. Suppose that $\frac{n+(n-1)\gamma}{n} < p < n+(n-1)\gamma$, $1 \le q < \frac{n p}{n + (n-1)\gamma}$, $r<\frac{nq}{n-q}$.
Then the first non-trivial Neumann $(p,r)$-eigenvalue
\begin{multline*}
\mu_{p,r}(\widetilde{\Omega}_{\widetilde{\gamma}}) \\
\geq {\left(|\mathbb{B}^n|^{1-\alpha} + |\mathbb{S}^{n-2}|^{\,1-\alpha} \,C_{\gamma}^{\alpha} \left[ \frac{1-\alpha}{\,n(1-\alpha) - \alpha (n-1)\gamma\,} \right]^{1-\alpha}\right)^{-\frac{1}{\alpha}}}\left(\mu_{q,r}(\mathbb B^n)\right)^{\frac{1}{\alpha}},\\ 0<\alpha=\frac{q}{p}<1.
\end{multline*}
}

\medskip

We remark that in \cite{BCT9, BCT15}, lower estimates involving the isoperimetric constant relative to $\Omega$ were obtained. The approach to spectral estimates in non-convex domains based on the geometric theory of composition operators on Sobolev spaces was given in \cite{GU16,GU17}. In the recent works \cite{GPU24, MU26}, the non-linear Neumann eigenvalue problems were studied in outward cuspidal domains.

\section{Sobolev spaces and composition operators }

\subsection{Sobolev spaces and composition operators }

Let $\Omega$ be a domain in the Euclidean space $\mathbb R^n$, $n\geq 2$. The Sobolev space $W^1_p(\Omega)$, $1\leq p\leq\infty$, is defined \cite{M}
as a Banach space of locally integrable weakly differentiable functions $f:\Omega\to\mathbb{R}$ equipped with the following norm: 
\[
\|f\|_{W^1_p(\Omega)}=\| f\|_{L_p(\Omega)}+\|\nabla f\|_{L_p(\Omega)},
\]
where $\nabla f$ is the weak gradient of the function $f$. In this article we work with the homogeneous seminormed Sobolev space $L^1_p(\Omega)$, $1\leq p\leq\infty$, with the following seminorm: 
\[
\|f\|_{L^1_p(\Omega)}=\|\nabla f\|_{L_p(\Omega)}.
\]

In accordance with the non-linear potential theory \cite{MH72} we consider elements of Sobolev spaces $W^1_p(\Omega)$ as classes of equivalence of quasicontinuous functions defined up to a set of $p$-capacity zero  \cite{M}. Recall that a function $f$ is termed quasicontinuous if for any $\varepsilon >0$ there is an open  set $U_{\varepsilon}$ such that the $p$-capacity of $U_{\varepsilon}$ is less than $\varepsilon$ and on the set $\Omega\setminus U_{\varepsilon}$ the function  $f$ is continuous (see, for example, \cite{HKM,M}).

Let $\Omega$ and $\widetilde{\Omega}$ be domains in $\mathbb R^n$. Then a homeomorphism $\varphi:\Omega\to \widetilde{\Omega}$ generates a bounded embedding operator
\[
\varphi^{\ast}:L^1_p(\widetilde{\Omega})\to L^1_q(\Omega),\,\,\,1\leq q\leq p\leq\infty,
\]
by the composition rule $\varphi^{\ast}(f)=f\circ\varphi$, if for
any function $f\in L^1_p(\widetilde{\Omega})$, the composition $\varphi^{\ast}(f)\in L^1_q(\Omega)$
is defined quasi-everywhere in $\Omega$ and there exists a constant $C_{p,q}(\varphi;\Omega)<\infty$ such that 
\[
\|\varphi^{\ast}(f)\|_{L^1_q(\Omega)}\leq C_{p,q}(\varphi;\Omega)\|f\|_{L^1_p(\widetilde{\Omega})}.
\]

In the geometric function theory composition operators on Sobolev spaces arise in the work \cite{VG75} and have numerous applications in the geometric analysis of PDE. Recall that the $p$-dilatation \cite{Ger69} of a Sobolev mapping $\varphi: \Omega\to \widetilde{\Omega}$ at a point $x\in\Omega$ is defined as
$$
K_p(x)=\inf \{k(x): |D\varphi(x)|\leq k(x) |J(x,\varphi)|^{\frac{1}{p}}\}.
$$

The following theorem, see \cite{GU25} and references therein, gives the characterization of composition operators in terms of integral characteristics of mappings of finite distortion. Recall that a weakly differentiable mapping $\varphi:\Omega\to\mathbb{R}^{n}$ is a mapping of finite distortion if $D\varphi(x)=0$ for almost all $x$ from $Z=\{x\in\Omega: J(x,\varphi)=0\}$ \cite{VGR}. 

\begin{thm}
\label{CompTh} Let $\varphi:\Omega\to\widetilde{\Omega}$ be a homeomorphism
between two domains $\Omega$ and $\widetilde{\Omega}$. Then $\varphi$ generates a bounded composition
operator 
\[
\varphi^{\ast}:L^1_p(\widetilde{\Omega})\to L^1_{q}(\Omega),\,\,\,1\leq q\leq p<\infty,
\]
 if and only if $\varphi\in W^1_{q,\loc}(\Omega)$, has finite distortion,
and 
\[
K_{p,q}(\varphi;\Omega) := \|K_p\|_{L_{\kappa}(\Omega)}<\infty, \,\,1/q-1/p=1/{\kappa}\,\,(\kappa=\infty, \text{ if } p=q).
\]
The norm of the operator $\varphi^\ast$ is estimated as $\|\varphi^\ast\| \leq K_{p,q}(\varphi;\Omega)$.
\end{thm}

This theorem in the case $p=q=n$ was proved in the work \cite{VG75}. The general case $1\leq q\leq p<\infty$ was proved in \cite{U93} where the weak change of variables formula \cite{H93} was used (see, also the case $n<q=p<\infty$ in \cite{V88}). 

In the case \(p=q=n\), the mappings that generate bounded composition operators on Sobolev spaces coincide with quasiconformal mappings \cite{VG75}.
By this reasons homeomorphisms $\varphi:\Omega\to\widetilde{\Omega}$ satisfying conditions of Theorem~\ref{CompTh} are called weak $(p,q)$-quasiconformal mappings \cite{VU98}. In the case $p=q$ such mappings are called weak $p$-quasiconformal mappings \cite{GGR95}.

In the case of weak $(p,q)$-quasiconformal mappings, the following composition duality theorem was proved in \cite{U93} (see, also \cite{GU19}).

\begin{thm}
\label{CompThD} Let a homeomorphism $\varphi:\Omega\to\widetilde{\Omega}$
between two domains $\Omega$ and $\widetilde{\Omega}$ generate a bounded composition
operator 
\[
\varphi^{\ast}:L^1_p(\widetilde{\Omega})\to L^1_{q}(\Omega),\,\,\,n-1<q \leq p< \infty.
\]
Then the inverse mapping $\varphi^{-1}:\widetilde{\Omega}\to\Omega$ generates a bounded composition operator 
\[
\left(\varphi^{-1}\right)^{\ast}:L^1_{q'}(\Omega)\to L^1_{p'}(\widetilde{\Omega}),
\]
where $p'=p/(p-n+1)$, $q'=q/(q-n+1)$. 
\end{thm}

\subsection{On $(p,q)$-composition reflections}

In this subsection we introduce the new notion of a $(p,q)$-composition reflection and prove that they are useful tools for constructions of bounded extension operators.

Let $\widetilde{\Omega} \subset \Omega \subset \mathbb{R}^n$ be Euclidean domains.  
Suppose there exists a homeomorphism  
\[
\varphi : \Omega \setminus \overline{\widetilde{\Omega}} \to \widetilde{\Omega}
\]
which is a weak $(p,q)$-quasiconformal mapping and admits a continuous extension to $\Omega \cap \partial \widetilde{\Omega}$ such that
\[
\varphi|_{\Omega \cap \partial \widetilde{\Omega}} = \mathrm{id}.
\]
Then $\varphi$ is called a \emph{$(p,q)$-composition reflection} of $\Omega \setminus \overline{\widetilde{\Omega}}$ onto $\widetilde{\Omega}$.

The following theorem provides the basic extension result via $(p,q)$-composition reflections.

\begin{thm}
\label{ext_pq}
Let $\widetilde{\Omega} \subset \Omega \subset \mathbb{R}^n$ be Euclidean domains such that 
$\Omega \cap \partial \widetilde{\Omega}$ is a Lipschitz surface. 
Suppose there exists a $(p,q)$-composition reflection 
$\varphi$ of $\Omega \setminus \overline{\widetilde{\Omega}}$ onto $\widetilde{\Omega}$, 
for some $1 < q \leq p < \infty$. 
Then there exists a bounded $(p,q)$-extension operator
\[
E: L^1_p(\widetilde{\Omega}) \to L^1_q(\Omega)
\]
satisfying
\begin{equation*}
\|E\|\leq \left(|\widetilde{\Omega}|^{\frac{p-q}{p}}+K_{p,q}^q(\varphi;\Omega\setminus\overline{\widetilde{\Omega}})\right)^{\frac{1}{q}}.
\end{equation*}
\end{thm}

\begin{proof}
Let $f \in L^1_p(\widetilde{\Omega})$ and define the composition operator $\varphi^{*} f := f \circ \varphi$. 
By the assumption, $\varphi : \Omega \setminus \overline{\widetilde{\Omega}} \to \widetilde{\Omega}$ is a $(p,q)$-composition reflection with 
$\varphi|_{\Gamma} = \mathrm{id}$ on $\Gamma := \Omega \cap \partial \widetilde{\Omega}$.
In particular,
\[
\varphi^{*}f \in L^1_q(\Omega \setminus \overline{\widetilde{\Omega}}),
\qquad
\text{and the traces } f|_{\Gamma} = (\varphi^{*}f)|_{\Gamma} \text{ coincide.}
\]

Define $\widetilde{f}$ on $\Omega$ by
\[
\widetilde{f}(x) :=
\begin{cases}
f(x), & x \in \widetilde{\Omega},\\
(\varphi^{*}f)(x), & x \in \Omega \setminus \overline{\widetilde{\Omega}}.
\end{cases}
\]

Since $\Gamma$ is a Lipschitz surface, the Sobolev trace operators are well-defined, and the above trace equality implies, by the standard gluing lemma \cite{evans2015measure}, that $\widetilde{f} \in L^1_q(\Omega)$ with
\[
\nabla \widetilde{f} =
\begin{cases}
\nabla f, & \text{a.e. on } \widetilde{\Omega},\\
\nabla (f \circ \varphi), & \text{a.e. on } \Omega \setminus \overline{\widetilde{\Omega}}.
\end{cases}
\]

Next we give estimates of the norm of this extension operator. Let $q<p$, then
\begin{multline*}
\|E(f)\|_{L^1_q(\Omega)}=\left(\|E(f)\|^q_{L^1_q(\Omega)}\right)^{\frac{1}{q}}=\left(\|E(f)\|^q_{L^1_q(\widetilde{\Omega})}+\|E(f)\|^q_{L^1_q(\Omega\setminus \overline{\widetilde{\Omega}})}\right)^{\frac{1}{q}}\\
\leq 
\left(|\widetilde{\Omega}|^{\frac{p-q}{p}}\|f\|^q_{L^1_p(\widetilde{\Omega})}+K_{p,q}^q(\varphi;\Omega\setminus\overline{\widetilde{\Omega}})\|f\|^q_{L^1_p(\widetilde{\Omega})}\right)^{\frac{1}{q}}
\\=
\left(|\widetilde{\Omega}|^{\frac{p-q}{p}}+K_{p,q}^q(\varphi;\Omega\setminus\overline{\widetilde{\Omega}})\right)^{\frac{1}{q}}\|f\|_{L^1_p(\widetilde{\Omega})}.
\end{multline*}
If $q=p$, then
\begin{multline*}
\|E(f)\|_{L^1_p(\Omega)}=\left(\|E(f)\|^p_{L^1_p(\Omega)}\right)^{\frac{1}{p}}=\left(\|E(f)\|^p_{L^1_p(\widetilde{\Omega})}+\|E(f)\|^p_{L^1_p(\Omega\setminus \overline{\widetilde{\Omega}})}\right)^{\frac{1}{p}}\\
\leq 
\left(\|f\|^p_{L^1_p(\widetilde{\Omega})}+K_{p,p}^p(\varphi;\Omega\setminus\overline{\widetilde{\Omega}})\|f\|^p_{L^1_p(\widetilde{\Omega})}\right)^{\frac{1}{p}}
=
\left(1+K_{p,p}^p(\varphi;\Omega\setminus\overline{\widetilde{\Omega}})\right)^{\frac{1}{p}}\|f\|_{L^1_p(\widetilde{\Omega})}.
\end{multline*}
\end{proof}

By using the composition duality theorem \cite{U93} we have the following corollary of the previous theorem.

\begin{cor}
Let $\widetilde{\Omega}\subset\Omega$ be Euclidean domains in $\mathbb R^n$ such that $\Omega \cap \partial \tilde{\Omega}$ is a Lipschitz surface in $\mathbb R^n$. Suppose that there exists a $(p,q)$-composition reflection $\varphi$ of $\Omega\setminus \overline{\widetilde{\Omega}}$ onto $\widetilde{\Omega}$, $n-1<q\leq p<\infty$. Then there exists a bounded $(q',p')$-extension operator
$$
E: L^1_{q'}(\Omega\setminus \overline{\widetilde{\Omega}}) \to L^1_{p'}(\Omega),\,\,p'=\frac{p}{p-n+1}\leq q'=\frac{q}{q-n+1}.
$$
\end{cor}

Recall the notion of uniform domains \cite{GO79} or domains with the flexible cone condition in another terminology \cite{B88}. Let $\Omega\subset\mathbb R^n$ be a domain. Then $\Omega$ is called a uniform domain if there exist a number $\varepsilon>0$  such that each
pair of points $x_1,x_2\in\Omega$ can be joined by a rectifiable arc $\gamma\in\Omega$ for which
$$
\begin{cases}
s(\gamma)\leq \varepsilon |x_1-x_2|,\\
\min\limits_{k=1,2}s(\gamma(x_k,x))\leq \varepsilon d(x,\partial \Omega)\,\,\text{for all}\,\,x\in\gamma.
\end{cases}
$$
In this notation $s(\gamma)$ denotes the Euclidean length of $\gamma$, $\gamma(x_k,x)$ is the part of $\gamma$ between $x_k$ and $x$,
and $d(x,\partial\Omega)$ denotes the Euclidean distance from $x$ to $\partial\Omega$.

\begin{cor}
Let $\widetilde{\Omega}\subset\Omega$ be Euclidean domains in $\mathbb R^n$ such that $\Omega \cap \partial \tilde{\Omega}$ is a Lipschitz surface in $\mathbb R^n$ and $\Omega$ is a uniform domain. Suppose that there exists a $(p,q)$-composition reflection $\varphi$ of $\Omega\setminus \overline{\widetilde{\Omega}}$ onto $\widetilde{\Omega}$, $1<q\leq p<\infty$. Then there exists a bounded $(p,q)$-extension operator
$$
E: L^1_p(\widetilde{\Omega}) \to L^1_{q}(\mathbb R^n).
$$
\end{cor}

\begin{proof}
Because there exists a $(p,q)$-composition reflection $\varphi$ of $\Omega\setminus \overline{\widetilde{\Omega}}$ onto $\widetilde{\Omega}$, then by Theorem~\ref{ext_pq} there exists a bounded $(p,q)$-extension operator
$$
E_{\widetilde{\Omega}}: L^1_p(\widetilde{\Omega}) \to L^1_{q}(\Omega).
$$
Now, since $\Omega$ is a uniform domain, by \cite{J81} there exists a bounded $(q,q)$-extension operator
$$
E_{\Omega}: L^1_q(\Omega) \to L^1_{q}(\mathbb R^n).
$$
Hence, we obtain the bounded extension operator $E$ as a composition of extension operators $E_{\Omega}\circ E_{\widetilde{\Omega}}$.
\end{proof}

\section{Extension operators in outward cuspidal domains}  

In this section we prove the existence of $(p,q)$-extension operators in outward cuspidal domains, refining the results of \cite{GS82} by using the theory of composition operators on Sobolev spaces \cite{U93}. We first consider the two-dimensional case, and subsequently, on this base, we present the multidimensional case.

\subsection{The two-dimensional case}  
Let us recall that the two-dimensional outward $\widetilde{\gamma}$-cuspidal domain 
$\widetilde{\Omega}_{\widetilde{\gamma}} \subset \mathbb{R}^2$, with $1<\widetilde{\gamma}<\infty$, is defined as 
\begin{equation}
\label{cusp_two}
\widetilde{\Omega}_{\widetilde{\gamma}} := \widetilde{\Omega}_{+} \,\cup\, \I \,\cup\, \widetilde{\Omega}_{-},
\end{equation}
where, in polar coordinates,
$$
\widetilde{\Omega}_{+} = \{(r,\theta): 0<r<1,\ 0<\theta< r^{\gamma}\}, \quad \text{with} \quad \gamma=\widetilde{\gamma}-1,
$$
and
$$
\widetilde{\Omega}_{-} = H(\widetilde{\Omega}_{+}), \quad \text{where} \quad H(x,y)=(x,-y), \quad \I := \{(x,0): 0<x<1\}.
$$

\begin{thm}
Let $\widetilde{\Omega}_{\widetilde{\gamma}} \subset \mathbb{D}$, where $\mathbb{D} \subset \mathbb{R}^n$ is the unit ball, be an outward $\widetilde{\gamma}$-cuspidal domain with $1<\widetilde{\gamma}<\infty$. Set $\gamma = \widetilde{\gamma} - 1$. Then, for any $1 < p < \infty$, the extension operator
\[
E : L^1_p(\widetilde{\Omega}_{\widetilde{\gamma}}) \to L^1_{q}(\mathbb{D})
\]
is bounded for all $1 \le q < \frac{2 p}{\gamma + 2}$, and satisfies
\[
\|E\| \le \left( \pi^{1-\alpha} + C_{\gamma}^{\alpha} \left[ \frac{1-\alpha}{\,2(1-\alpha) - \alpha\gamma\,} \right]^{1-\alpha} \right)^{\frac{1}{q}},
\]
with $\frac{1}{p} \le \alpha < \frac{2}{\gamma + 2}$ arbitrary and
$$
C_{\gamma}=\pi  \left(1+(\gamma+1)^2\frac{\pi^4}{(\pi-1)^4}+\frac{1}{(\pi-1)^2}\right)^{\frac{p}{2}}.
$$
\end{thm}

\begin{proof}
Define the natural reflection 
$$
\varphi :\mathbb D\setminus \overline{\widetilde{\Omega}} \to \widetilde{\Omega},
$$ 
as an inversion by $\theta$  on the circles $S(0,r)$, $0<r<1$. We write explicitly this inversion for the upper half disc. Because symmetry of $\theta$ all necessary calculation  on the upper half disc are relevant for the lower half disc also.

By using calculations of \cite{GS82} and the substitution $s=\theta r$, this reflection is the following
$$
\varphi(r,s)=(R(r,s),S(r,s))=\left(r, r^{\gamma}\cdot \frac{\pi r-s}{\pi -r^{\gamma}}\right).
$$
The coordinate system $(r,s)$ is an infinitesimally orthonormal coordinate system. It means that the value of the norm of the differential $D\varphi(r,s)$ and the value of the Jacobian  $J(\varphi,(r,s))$ coincide with their values in the Euclidean coordinates.
The partial derivatives:
\begin{multline*}
\frac{\partial R(r,s)}{\partial r}=1,\,\,\frac{\partial R(r,s)}{\partial s}=0,\\
\frac{\partial S(r,s)}{\partial r}=\gamma r^{\gamma-1}\cdot \frac{\pi r-s}{\pi -r^{\gamma}}+
r^{\gamma}\left(\frac{\pi(\pi-r^{\gamma})+\gamma r^{\gamma-1}(\pi r-s)}{(\pi-r^{\gamma})^2}\right),
\\
\frac{\partial S(r,s)}{\partial s}=-\frac{r^{\gamma}}{\pi-r^{\gamma}}.
\end{multline*}
Then
\begin{multline*}
\left|\frac{\partial R(r,s)}{\partial r}\right|=1,\,\,\left|\frac{\partial R(r,s)}{\partial s}\right|=0,\\
\left|\frac{\partial S(r,s)}{\partial r}\right|\leq
\frac{\pi\gamma}{\pi-1}+\frac{\pi^2+\pi\gamma}{(\pi-1)^2}=(\gamma+1)\frac{\pi^2}{(\pi-1)^2},\\
\left|\frac{\partial S(r,s)}{\partial s}\right|\leq \frac{1}{\pi-1}.
\end{multline*}
Hence, we have the following estimate of the differential $D\varphi$
$$
1<|D\varphi (s,r)|<\left(1+(\gamma+1)^2\frac{\pi^4}{(\pi-1)^4}+\frac{1}{(\pi-1)^2}\right)^{\frac{1}{2}}<\infty.
$$

Introduce a correction factor $\alpha:=q/p<1$. Then $\varphi :\mathbb D\setminus \overline{\widetilde{\Omega}} \to \widetilde{\Omega}$ be a $(p,q)$-composition reflection, if
$$
K_{p,q}^{\frac{p\alpha}{1-\alpha}}(\varphi)=\int\limits_{\mathbb D \setminus \overline{\widetilde{\Omega}}}
 \left(\frac{|D \varphi (s,r)|^p}{|J(\varphi, (s,r))|}\right)^{\frac{\alpha}{1-\alpha}}~drds<\infty.
$$  
Convergence of the integral at the the origin of coordinates depends only on 
$$
|J(\varphi, (s,r))|=\frac{\partial R(r,s)}{\partial r}\frac{\partial S(r,s)}{\partial s}=\frac{r^{\gamma}}{\pi-r^{\gamma}}\geq \frac{r^{\gamma}}{\pi}.
$$
Hence, combining these estimates we obtain 
\begin{multline*}
K_{p,q}^{q}(\varphi;\mathbb D \setminus \overline{\widetilde{\Omega}})
=
K_{p,q}^{\alpha p}(\varphi;\mathbb D \setminus \overline{\widetilde{\Omega}})=\left(\int\limits_{\mathbb D \setminus \overline{\widetilde{\Omega}}}
 \left(\frac{|D \varphi (s,r)|^p}{|J(\varphi, (s,r))|}\right)^{\frac{\alpha}{1-\alpha}}~drds\right)^{1-\alpha}\\
\leq
\left(\int\limits_{\mathbb D} \left(\frac{|D \varphi (s,r)|^p}{|J(\varphi, (s,r))|}\right)^{\frac{\alpha}{1-\alpha}}~drds\right)^{1-\alpha}
\leq 
\left(\int\limits_{\mathbb D} \left(\frac{C_{\gamma}}{r^{\gamma}} \right)^{\frac{\alpha}{1-\alpha}}~drds\right)^{1-\alpha}
\\=
\left(\int_0^1  \left(\frac{C_{\gamma}}{r^{\gamma}} \right) ^{\frac{\alpha}{1-\alpha}}~dr\int\limits_0^{2\pi r}~ds\right)^{1-\alpha}
=(2\pi)^{1-\alpha}C_{\gamma}^{\alpha}\left(\int_0^1  \frac{r}{r^{\frac{\alpha\gamma}{1-\alpha}}} ~dr\right)^{1-\alpha},
\end{multline*} 
where 
$$
C_{\gamma}=\pi  \left(1+(\gamma+1)^2\frac{\pi^4}{(\pi-1)^4}+\frac{1}{(\pi-1)^2}\right)^{\frac{p}{2}}.
$$

This integral converges if 
$$
\frac{\alpha \gamma}{1-\alpha}-1<1,
$$ 
i.e. for any $\alpha<\frac{2}{\gamma+2}$.
Because $1 \leq \alpha p$ we proved that $\varphi$ is the $(p, \alpha p)$-composition reflection for any 
$$
\frac{1}{p} \leq \alpha < \frac{2}{\gamma+2}.
$$

Hence, by Theorem~\ref{ext_pq} there exists a bounded $(p,q)$-extension operator
$$
E: L^1_p(\widetilde{\Omega}) \to L^1_{q}(\mathbb D), \,\,1\leq q<\frac{2}{\gamma+2}p,
$$
with 
$$
\|E\|\leq \left(\pi^{1-\alpha}+C_{\gamma}^{\alpha}\left(\frac{1-\alpha}{2(1-\alpha) -\alpha\gamma}\right)^{1-\alpha}\right)^{\frac{1}{q}}.
$$

\end{proof}

\begin{rem}
In the terminology of the works \cite{KZ22,KZ24} the considered domain is the domains with the $\frac{1}{\widetilde{\gamma}}$-H\"older cusp and the so these extension operators correspond to extension operators constructed in \cite{KZ22,KZ24,MP86,MP87}. The suggested method allows to obtain estimates of the operator norms, that is critical for applications of extension operators in the spectral estimates of elliptic operators \cite{GPU20}.
\end{rem}

\subsection{The multidimensional case. Picks.}

Let us define an outward $\widetilde{\gamma}$-cuspidal domain $\widetilde{\Omega}_{\widetilde{\gamma}}\subset\mathbb R^n$ 
as the domain obtained by the rotation of the two-dimensional domain
$$
\widetilde{\Omega}_+ := \{(r,\theta_1): 0< \theta_1 < r^{\gamma}, \ 0<r<1\},\quad \gamma=\widetilde{\gamma}-1,
$$
around the axis $x_1$, where $(r, \theta_1, \theta_2, \ldots, \theta_{n-1})$ are spherical coordinates.

\begin{thm}
Let $\widetilde{\Omega}_{\widetilde{\gamma}} \subset \mathbb{B}^n$, where $\mathbb{B}^n \subset \mathbb{R}^n$ is the unit ball, be an outward $\widetilde{\gamma}$-cuspidal domain with $1<\widetilde{\gamma}<\infty$. Set $\gamma = \widetilde{\gamma} - 1$. Then, for any $\frac{n+(n-1)\gamma}{n} < p < \infty$, the extension operator
\[
E : L^1_p(\widetilde{\Omega}_{\widetilde{\gamma}}) \to L^1_{q}(\mathbb{B}^n)
\]
is bounded for all $1 \le q < \frac{n p}{n + (n-1)\gamma}$, and satisfies
\[
\|E\| \le \left( |\mathbb{B}^n|^{1-\alpha} + |\mathbb{S}^{n-2}|^{\,1-\alpha} \,C_{\gamma}^{\alpha} \left[ \frac{1-\alpha}{\,n(1-\alpha) - \alpha (n-1)\gamma\,} \right]^{1-\alpha} \right)^{\frac{1}{q}},
\]
with $\frac{1}{p} \le \alpha < \frac{n}{n + (n-1)\gamma}$ arbitrary and
$$
C_{\gamma} = \pi \left(1 + (\gamma+1)^2 \frac{\pi^4}{(\pi-1)^4} + \frac{1}{(\pi-1)^2}\right)^{\frac{p}{2}}.
$$
\end{thm}

\begin{proof} We will use the spherical coordinates
$(r, \theta_1, \theta_2, \ldots, \theta_{n-1})$. The corresponding reflection in the upper half-space $x_n>0$ is similar to the case $n=2$:
$$
\varphi_n(r,s,\theta_2, \ldots, \theta_{n-1}) = \left(r, r^{\gamma} \cdot \frac{\pi r - s}{\pi - r^{\gamma}}, \theta_2, \ldots, \theta_{n-1}\right),
$$
where $s = \theta_1 r$. 

Using the reflection $H(x_1, \ldots, x_n) = (x_1, \ldots, -x_n)$ we obtain, as in the two-dimensional case, a natural reflection 
$$
\varphi_n : \mathbb{B}^n \setminus \overline{\widetilde{\Omega}}_{\widetilde{\gamma}} \to \widetilde{\Omega}_{\widetilde{\gamma}}.
$$
Using the previous two-dimensional calculations we obtain
\begin{multline*}
K_{p,q}^{q}(\varphi_n; \mathbb{B}^n \setminus \overline{\widetilde{\Omega}}_{\widetilde{\gamma}}) 
= K_{p,q}^{\alpha p}(\varphi_n; \mathbb{B}^n \setminus \overline{\widetilde{\Omega}}_{\widetilde{\gamma}}) \\
=\left(\int\limits_{\mathbb{B}^n \setminus \overline{\widetilde{\Omega}}_{\widetilde{\gamma}}}
\left(\frac{|D \varphi_n (s,r,\theta_2,\ldots, \theta_{n-1})|^p}{|J(\varphi_n, (s,r,\theta_2,\ldots, \theta_{n-1}))|}\right)^{\frac{\alpha}{1-\alpha}}
\,dr\,ds\,d\theta_2 \ldots d\theta_{n-1}\right)^{1-\alpha} \\
\leq |\mathbb{S}^{n-2}|^{\,1-\alpha} \, C_{\gamma}^{\alpha} \left(\int_0^1  \frac{r^{n-1}}{r^{\frac{\alpha (n-1)\gamma}{1-\alpha}}} \,dr\right)^{1-\alpha},
\end{multline*} 
where $|\mathbb{S}^{n-2}|$ is the $(n-2)$-dimensional surface measure of the unit sphere.

The difference with the two-dimensional case is the power $n-1$ in the exponent of $\gamma$, corresponding to the change from Euclidean to spherical coordinates and the $(n-1)$ tangential directions. This integral converges if 
$$
\frac{\alpha (n-1)\gamma}{1-\alpha} - (n-1) < 1 \quad \Longleftrightarrow \quad \alpha < \frac{n}{n + (n-1)\gamma}.
$$

Because $1 \leq \alpha p$, we have proved that $\varphi_n$ is a $(p, \alpha p)$-composition reflection for any 
$$
\frac{1}{p} \leq \alpha < \frac{n}{n + (n-1)\gamma}.
$$

Hence, by Theorem~\ref{ext_pq} there exists a bounded $(p,q)$-extension operator
$$
E: L^1_p(\widetilde{\Omega}_{\widetilde{\gamma}}) \to L^1_{q}(\mathbb{B}^n), \quad 1\leq q < \frac{n p}{n + (n-1)\gamma},
$$
with 
$$
\|E\| \leq \left(|\mathbb{B}^n|^{1-\alpha} + |\mathbb{S}^{n-2}|^{\,1-\alpha} \,C_{\gamma}^{\alpha} \left(\frac{1-\alpha}{n(1-\alpha) - \alpha (n-1)\gamma}\right)^{1-\alpha}\right)^{\frac{1}{q}}.
$$
\end{proof}

\subsection{The multidimensional case. Ridges.}

Let $\Omega_{n,r} = \mathbb{B}^{\,n-1} \times (0,1)$, where $\mathbb{B}^{\,n-1} \subset \mathbb{R}^{\,n-1}$ is the unit ball, and 
$\widetilde{\Omega}_{n,r} = \widetilde{\Omega}_{n-1} \times (0,1)$, where $\widetilde{\Omega}_{n-1}$ is the $(n-1)$-dimensional outward cuspidal domain from the previous theorem. 
The calculations are similar to the previous case and correspond to the case with $\mathbb{B}^{\,n-1}$.

Hence, there exists a bounded $(p,q)$-extension operator
$$
E: L^1_p(\widetilde{\Omega}_{n,r}) \to L^1_{q}(\Omega_{n,r}), 
\qquad 
1 \le q < \frac{(n-1)p}{(n-1) + (n-2)\gamma},
$$
for any 
$$
\frac{1}{p} \le \alpha < \frac{n-1}{(n-1) + (n-2)\gamma},
$$
with 
$$
\|E\| \le 
\left(
|\Omega_{n,r}|^{1-\alpha} 
+ 
|\mathbb{S}^{\,n-3}|^{\,1-\alpha} 
\, C_{\gamma}^{\alpha} 
\left(
\frac{1-\alpha}{(n-1)(1-\alpha) - \alpha (n-2)\gamma}
\right)^{1-\alpha}
\right)^{\frac{1}{q}},
$$
where $|\mathbb{S}^{\,n-3}|$ is the $(n-3)$-dimensional surface measure of the unit sphere in $\mathbb{R}^{\,n-2}$.

\section{Existence of $(p,q)$-composition reflections}  

By using the composition duality property \cite{U93} the Liouville type theorems were obtained for weak $(p,q)$-quasi\-con\-for\-mal mappings in \cite{GU19}. The following  Liouville type theorem partially explains the nonexistence part of Theorem~1.2 and Theorem~1.3 from \cite{KZ22}. Denote by $\overline{\mathbb R^n}$ the one-point compactification of the Euclidean space $\mathbb R^n$.

\begin{thm}
\label{CompRef} Let $\widetilde{\Omega}$ be a bounded domain in $\overline{\mathbb R^n}$. Then there no exists a weak $(p,q)$-quasiconformal reflection $\varphi:\overline{\mathbb R^n} \setminus\overline{\widetilde{\Omega}}\to \widetilde{\Omega}$, $n-1<q\leq p<n$.  
\end{thm}

\begin{proof}
We prove this theorem by contradiction. Suppose that there exists a weak $(p,q)$-quasiconformal reflection $\varphi:\overline{\mathbb R^n} \setminus\overline{\widetilde{\Omega}}\to \widetilde{\Omega}$, $n-1<q\leq p<n$. Then the composition operator
$$
\varphi^{\ast}: L^1_p(\widetilde{\Omega}) \to L^1_q(\overline{\mathbb R^n} \setminus\overline{\widetilde{\Omega}})
$$
is bounded. By the composition duality theorem the composition operator
$$
\left(\varphi^{-1}\right)^{\ast}: L^1_{q'}(\overline{\mathbb R^n} \setminus\overline{\widetilde{\Omega}}) \to L^1_{p'}(\widetilde{\Omega}) 
$$
is bounded also. 

By the conditions of the theorem we have:
$$
n-1<q\leq p<n<p'=\frac{p}{p-(n-1)}<q'=\frac{q}{q-(n-1)}<\infty.
$$
Because $\widetilde{\Omega}$ is the bounded domain, then 
$$
L^1_{p'}(\widetilde{\Omega}) \hookrightarrow  L^1_p(\widetilde{\Omega}), 
$$
and so
$$
 L^1_{q'}(\overline{\mathbb R^n} \setminus\overline{\widetilde{\Omega}}) \hookrightarrow L^1_q(\overline{\mathbb R^n} \setminus\overline{\widetilde{\Omega}}).
$$
Hence the Lebesgue measure $|\overline{\mathbb R^n} \setminus\overline{\widetilde{\Omega}}|<\infty$. Contradiction.
\end{proof}

\medskip

\textbf{On $(p,p')$-composition reflections in $\mathbb R^2$.}  
Recall the following composition duality property \cite{U93}.

\begin{thm}
\label{dualcomp}
The homeomorphism $\varphi:\Omega\to\widetilde{\Omega}$,  $\Omega,\widetilde{\Omega}\subset\mathbb R^2$, induces a bounded composition
operator 
\[
\varphi^{\ast}:L^1_p(\widetilde{\Omega})\to L^1_q(\Omega),\,\,\,1<q\leq p<\infty,
\]
if and only if the inverse mapping $\varphi^{-1}:\widetilde{\Omega}\to\Omega$ induces a bounded composition operator 
\[
\left(\varphi^{-1}\right)^{\ast}:L^1_{q'}(\Omega)\to L^1_{p'}(\widetilde{\Omega}),\,\,\,1<p'=\frac{p}{p-1}\leq q'=q/(q-1)<\infty.
\]
\end{thm}

Hence. in the two-dimensional case we have the following duality extension property, as a consequence of Theorem~\ref{dualcomp}.

\begin{thm} 
\label{dual} Let $\widetilde{\Omega}\subset\mathbb R^2$, be a Euclidean domain. Then there exists a $(p,p')$-composition reflection $\varphi$ of $\Omega=\overline{\mathbb R^2}\setminus \overline{\widetilde{\Omega}}$ onto $\widetilde{\Omega}$, $1<q\leq p<\infty$, if and only if there exists a $(p,p')$-composition reflection $\varphi^{-1}$ of $\widetilde{\Omega}$ onto  $\Omega=\overline{\mathbb R^2}\setminus \overline{\widetilde{\Omega}}$. 
\end{thm}

Let us recall that a function $\Phi$ defined on the class of open subsets of $\mathbb R^n$ and taking nonnegative values is called a quasiadditive set function (see, for example, \cite{VU04}) if, for all open sets $U_1\subset U_2\subset\mathbb R^n$, we have
$$
\Phi(U_1)\leq \Phi(U_2),
$$
and there exists a positive constant $\theta$ such that for every collection of pairwise disjoint open sets $\left\{U_k\subset\mathbb R^n\right\}_{k\in\mathbb N}$ we have
$$
\sum\limits_{k=1}^{\infty}\leq \theta \Phi\left(\bigcup_{k=1}^{\infty}U_k\right).
$$

By using the generalized Ahlfors condition \cite{KUZ22}, we obtain the following necessary condition of existence of $(p,p')$-composition reflections.

\begin{thm}
\label{ahl} Let $\widetilde{\Omega}\subset\mathbb R^2$, be a Euclidean domain and there exists a $(p,p')$-composition reflection $\varphi$ of $\Omega=\overline{\mathbb R^2}\setminus \overline{\widetilde{\Omega}}$ onto $\widetilde{\Omega}$, $1<p<\infty$. Then 
$$
\Phi_i(B(x,r))^{p-p'}|B(x,r)\cap\Omega|^{p'}\geq c_0 |B(x,r)|^p, \,\,0<r< 1,
$$
and 
$$
\Phi_o(B(x,r))^{p-p'}|B(x,r)\cap(\mathbb R^2\setminus\Omega)|^{p'}\geq c_0 |B(x,r)|^p, \,\,0<r< 1,
$$
where $\Phi_i$ and $\Phi_o$ are quasiadditive set functions associated with the extension operators, $1/p+1/p'=1$.
\end{thm}

\section{Extension operators on $p$-Ahlfors domains}

In this subsection we introduce the notion of $p$-Ahlfors domains and use it for constructions of bounded extension operators.

In the construction of $p$-Ahlfors domains we use the notion of $p$-quasiconformal mappings. Recall that a mapping $\varphi:\Omega\to\widetilde{\Omega}$ is called a weak $p$-quasiconformal mapping \cite{GGR95,U93} if $\varphi\in W^1_{p,\loc}(\Omega)$, has finite distortion and
$$
K_{p,p}^{p}(\varphi;\Omega)=\ess\sup\limits_{\Omega}\frac{|D\varphi(x)|^p}{|J(x,\varphi)|}<\infty.
$$
The following example refines the example from \cite{GGR95} and demonstrates that the behaviour of $p$-quasiconformal mappings ($p \neq n$) can be essentially different from the behaviour of usual quasiconformal mappings.

\vskip 0.2cm
\noindent
{\bf Examples of $p$-quasiconformal mappings.}  Let $ \vec{\alpha}=(\alpha_1,\alpha_2,\dots ,\alpha_n)$ be a multiindex, $-1< \alpha_1 \leq \alpha_2 \leq \dots \leq \alpha_n$, and $\alpha=\sum\limits_{i=1}^n{\alpha_i}$. Consider the homeomorphism $\varphi:\mathbb B^n \to \mathbb R^n$ defined by the following way: 
$$
\varphi(x)=(x_1\cdot |x|^{\alpha_1},x_2\cdot |x|^{\alpha_2},\dots, x_n\cdot|x|^{\alpha_n}).
$$
Then
$$
\frac{\partial\varphi_i}{\partial x_j}=\left(1+\alpha_i \frac{x_i^2}{|x|^2}\right)|x|^{\alpha_i},\,\,\text{if}\,\,i= j,\,\,\frac{\partial\varphi_i}{\partial x_j}=0, \,\,\text{if}\,\,i\ne j.
$$
Because $\alpha_1 \leq \alpha_2 \leq \dots \leq \alpha_n$ we have 
$$
|D\varphi(x)|=\left(1+\alpha_1 \frac{x_1^2}{|x|^2}\right)|x|^{\alpha_1}\leq \left(1+\alpha_1 \right)|x|^{\alpha_1}.
$$
Now we estimate the Jacobain of $\varphi$:
\begin{multline*}
|J(x,\varphi)|=\prod\limits_{i=1}^n \left(1+\alpha_i \frac{x_i^2}{|x|^2}\right)|x|^{\alpha_i}=|x|^{\alpha}\prod\limits_{i=1}^n \left(1+\alpha_i \frac{x_i^2}{|x|^2}\right)\\
\geq |x|^{\alpha} \left(1+\sum\limits_{i=1}^n\alpha_i \frac{x_i^2}{|x|^2}\right)\geq |x|^{\alpha} \left(1+\sum\limits_{i=1}^n\alpha_i \right)=|x|^{\alpha}(1+\alpha).
\end{multline*}
Hence
$$
\ess\sup\limits_{\Omega}\frac{|D\varphi(x)|^p}{|J(x,\varphi)|}\leq \frac{\left(1+\alpha_1 \right)^p|x|^{\alpha_1 p}}{|x|^{\alpha} \left|1+\alpha \right|}=\frac{\left(1+\alpha_1 \right)^p}{|1+\alpha|}\cdot
|x|^{\alpha_1 p-\alpha}.
$$

Now we consider three cases: $ \alpha_1 > 0$, $\alpha  \alpha_1^{-1}=n$ and $ \alpha_1 < 0$. 

\vskip 0.2cm
\noindent
1) Let $\alpha_1 > 0 $ then $ p \geq \alpha \cdot \alpha_1^{-1}$ and $\alpha \alpha_1^{-1} \geq n$. So the homeomorphism
 $\varphi$ is $p$-quasiconformal for all $p \geq  \alpha  \alpha_1^{-1}$. If $p  <\alpha  \alpha_1^{-1}$ the homeomorphism $\varphi$ is not
$p$-quasiconformal.

\vskip 0.2cm
\noindent
2) Let $\alpha  \alpha_1^{-1} = n$ then $\alpha_1 = \alpha_i$ for all $i=1,2.\dots n$ and the homeo\-mor\-phism $\varphi$ is
quasi\-con\-for\-mal.

\vskip 0.2cm
\noindent
3) Let $ \alpha \leq \alpha_1 < 0$ then $\varphi$ is $p$-quasiconformal for all  $p \in [1, \alpha \alpha_1^{-1}]$. In this case signs of the numbers $\alpha_1, \alpha_2, \dots , \alpha_n$ can be different. We can choose the numbers such that $\alpha_1 < \alpha_2 < \dots  < \alpha_{n-1} < 0$,  $\alpha_n > 0$ and $\alpha \leq \alpha_1$. In this case the ho\-meo\-mor\-phism $\varphi$ is $p$-quasiconformal for $1 \leq p < \alpha \alpha_1^{-1} < n-1$ and both  homeomorphisms $\varphi$, $\varphi^{-1}$ do not possess the Lipschitz property. 

\subsection{Extension operators on $p$-Ahlfors domains in $\mathbb R^n$}

Let $\widetilde{\Omega}_p \subset \mathbb R^n$ be a bounded domain. Then $\widetilde{\Omega}_p$ is called a $p$-Ahlfors domain if there exists a $p$-quasiconformal mapping 
$$
\varphi_p: \mathbb R^n\to\mathbb R^n
$$
such that $\varphi_p$ maps the unit disc $\mathbb B^n$ onto $\widetilde{\Omega}_p$, i.e. $\widetilde{\Omega}_p=\varphi_p(\mathbb B^n)$.

\begin{thm}
\label{ext_space} Let $\widetilde{\Omega}_p\subset\mathbb R^n$ be a $p$-Ahlfors domain, $n\leq p<\infty$. Then there exist a Sobolev extension operator
$$
E: L^1_p(\widetilde{\Omega}_p) \to L^1_{p'}(\mathbb R^n), \,\,p'=\frac{p}{p-n+1}.
$$
\end{thm}

\begin{proof}
Since $\widetilde{\Omega}_p$ is a $p$-Ahlfors domain in the Euclidean space $\mathbb R^n$, there exists a $p$-quasiconformal mapping 
$$
\varphi_p: \mathbb R^n\to\mathbb R^n
$$
such that $\widetilde{\Omega}_p=\varphi_p(\mathbb B^n)$. This $p$-quasiconformal mapping generates a bounded composition operator
$$
\varphi_p^{\ast}: L^1_p(\widetilde{\Omega}_p)\to L^1_p(\mathbb B^n).
$$
Now let a function $f\in L^1_p(\widetilde{\Omega}_p)$. Then the composition $g=\varphi_p^{\ast}(f)\in L^1_p(\mathbb B)$ and so $g=\varphi_p^{\ast}(f)$ belongs to the Sobolev space $L^1_{p'}(\mathbb B^n)$, $p'<n<p$, since the unit ball has finite measure, and we have the continuous embedding $L^1_p(\mathbb B^n) \hookrightarrow L^1_{p'}(\mathbb B^n)$, because $p'<p$.

Since $\mathbb B^n$ is the unit ball, then by \cite{C61,S70} there exists a bounded extension operator
$$
E_{\mathbb B^n}: L^1_{p'}(\mathbb B^n) \to L^1_{p'}(\mathbb R^n).
$$
Denote the extended function by $\widetilde{g}=E_B(g)\in L^1_{p'}(\mathbb R^n)$. By the composition duality property (Theorem~\ref{CompThD}) there exists a $p'$-quasiconformal mapping 
$$
\varphi_{p'}: \mathbb R^n\to\mathbb R^n, \,\, \varphi_{p'}=\varphi^{-1}_p,
$$
such that $\mathbb R^n\setminus \mathbb B^n=\varphi_{p'}(\mathbb R^n\setminus \widetilde{\Omega}_p)$. This $p'$-quasiconformal mapping generates a bounded composition operator
$$
\varphi_{p'}^{\ast}: L^1_{p'}(\mathbb R^n)\to L^1_{p'}(\mathbb R^n).
$$

Let a function $g\in L^1_{p'}(\mathbb R^n)$. Then the composition $\widetilde{f}=\varphi_{p'}^{\ast}(g)\in L^1_{p'}(\mathbb R^n)$ and
$$
\widetilde{f}\big|_{\Omega_p}=f\circ\varphi_p\circ\varphi_{p'}=f\circ\varphi_p\circ\varphi_{p}^{-1}=f.
$$
Hence we construct a bounded extension operator
$$
E: L^1_p(\widetilde{\Omega}_p) \to L^1_{p'}(\mathbb R^n), \,\,p'=\frac{p}{p-n+1},
$$
as a composition of the composition operator $\varphi^{\ast}$, the classical extension operator from the ball $E_B$ and the inverse composition operator $\varphi_{p'}^{\ast}$.
\end{proof}

\section{On Neumann $(p,q)$-eigenvalue problem}

In this section, we consider the non-linear Neumann $(p,q)$-eigenvalue problem \cite{GPU24}, which can be formulated in bounded Lipschitz domains $\Omega\subset\mathbb R^n$ as follows:
\begin{equation}\label{maineqn}
-\Delta_p u:=-\text{div}(|\nabla u|^{p-2}\nabla u)=\mu\|u\|_{L_q(\Omega)}^{p-q}|u|^{q-2}u\text{ in }\Omega,\quad\frac{\partial u}{\partial\nu}=0\text{ on }\partial \Omega.
\end{equation}

We study this  Neumann $(p,q)$-eigenvalue problem in the weak formulation:
\begin{equation}
\label{mainweak}
\int\limits_{\Omega}|\nabla u|^{p-2}\nabla u\cdot \nabla w\,dx=\lambda\|u\|_{L_{q}(\Omega)}^{p-q}\int\limits_{\Omega}|u|^{q-2}u\cdot w\,dx,
\end{equation}
for any function $w\in W^1_p(\Omega)$, which can be posed in  bounded non-Lipschitz domains $\Omega\subset\mathbb R^n$.

In \cite{GPU24} it was proved that the first non-trivial Neumann $(p,q)$-eigenvalues can be characterized by the the Min-Max Principle
$$
\mu_{p,q}(\Omega)\\=\min \left\{\frac{\|\nabla u\|_{L_p(\Omega)}^p}{\|u\|_{L_q(\Omega)}^p} :
u \in W^{1}_{p}(\Omega) \setminus \{0\}, \int_{\Omega} |u|^{q-2}u~dx=0 \right\}.
$$

Moreover, $\mu_{p,q}(\Omega)^{-\frac{1}{p}}$ is the best constant $B_{p,q}(\Omega)$ in the following Poincar\'e-Sobolev inequality 
$$
\inf\limits_{c\in\mathbb R}\|f-c\|_{L_q(\Omega)}\leq B_{p,q}(\Omega)\|\nabla f \|_{L_p(\Omega)},\,\,\, f\in W^{1}_{p}(\Omega).
$$

In the following theorem we give the weak monotonicity property of Neumann eigenvalues in Sobolev extension domains.

\begin{thm} 
\label{thm:mon}
Let $\widetilde{\Omega}\subset \mathbb R^n$ be a $(p,q)$-extension domain. Suppose that in a bounded domain $\Omega\supset\widetilde{\Omega}$ there exists the compact embedding operator $i:W^1_q(\Omega)\hookrightarrow L_r(\Omega)$.  Then
$$
\mu_{p,r}(\widetilde{\Omega}) \geq \frac{1}{\|E_{\Omega}\|^p} \cdot \left(\mu_{q,r}(\Omega)\right)^{\frac{p}{q}},
$$
where $\|E_{\Omega}\|$ denotes the norm of a bounded linear extension operator 
$$
E_{\Omega}: L^1_p(\widetilde{\Omega})\to L^1_q(\Omega).
$$
\end{thm}

\begin{proof}
Because $\widetilde{\Omega}$ is the $(p,q)$-extension domain then  there exists a continuous linear extension operator 
\begin{equation}\label{ExtOper}
E_{\Omega}:L_p^1(\widetilde{\Omega}) \to L_q^1({\Omega})
\end{equation}
defined by formula
\[ 
(E_{\Omega}(u))(x) = \begin{cases}
u(x) & \text{if}\,\, x \in \widetilde{\Omega}, \\
\widetilde{u}(x) & \text{if}\,\, x \in {\Omega} \setminus \widetilde{\Omega},
\end{cases} 
\]
where $\widetilde{u}:{\Omega} \setminus \widetilde{\Omega} \to \mathbb R$ be an extension of the function $u$.

Hence for every function $u \in W_p^1(\widetilde{\Omega})$ we have the following estimates
\begin{multline}\label{ineq1}
\|u-u_{\widetilde{\Omega}}\|_{L_r(\widetilde{\Omega})} 
{} =\inf\limits_{c\in \mathbb R}\|u-c\|_{L_r(\widetilde{\Omega})}=\inf\limits_{c\in \mathbb R}\|E_{\Omega}u-c\|_{L_r(\widetilde{\Omega})} \\
{} \leq \|E_{\Omega}u-(E_{\Omega}u)_{{\Omega}}\|_{L_r(\widetilde{\Omega})} \leq \|E_{\Omega}u-(E_{\Omega}u)_{{\Omega}}\|_{L_r({\Omega})},
\end{multline}
where $u_{\widetilde{\Omega}}$ and $(E_{\Omega}u)_{\Omega}$ are mean values of corresponding functions $u$ and $E_{\Omega}u$:
$$
u_{\widetilde{\Omega}}=\frac{1}{|\widetilde{\Omega}|}\int\limits_{\widetilde{\Omega}} |u|^{r-2}u~dx,\,\,(E_{\Omega}u)_{\Omega}=\frac{1}{|{\Omega}|}\int\limits_{{\Omega}} |E_{\Omega}u|^{r-2}E_{\Omega}u~dx.
$$

Because there exists the compact embedding operator $i:W^1_q(\Omega)\hookrightarrow L_r(\Omega)$, then taking into account the Poincar\'e-Sobolev inequality \cite{M}
$$
\|E_{\Omega}u-(E_{\Omega}u)_{{\Omega}}\|_{L_r(\Omega)} \leq B_{q,r}(\Omega) \|\nabla \left(E_{\Omega}u\right) \|_{L_q(\Omega)}, \quad u \in W_q^1(\Omega), 
$$
and the continuity of the linear extension operator \eqref{ExtOper}, i.e.,
$$
\|E_{\Omega}u \|_{L_q^1(\Omega)} \leq \|E_{\Omega}\| \cdot \|u \|_{L_p^1(\widetilde{\Omega})}, 
$$ 
we obtain
\begin{multline}\label{ineq2}
\|E_{\Omega}u-(E_{\Omega}u)_{{\Omega}}\|_{L_r(\Omega)} \leq B_{q,r}(\Omega) \|\nabla \left(E_{\Omega}u\right) \|_{L_q(\Omega)} \\
\leq B_{q,r}(\Omega) \cdot \|E_{\Omega}\| \cdot \|\nabla u \|_{L_p(\widetilde{\Omega})}.
\end{multline}

Combining inequalities \eqref{ineq1} and \eqref{ineq2} we have
$$
\|u-u_{\widetilde{\Omega}}\|_{L_r(\widetilde{\Omega})}  \leq B_{p,r}(\widetilde{\Omega}) \cdot \|\nabla u \|_{L_p(\widetilde{\Omega})},
$$
where 
$$
B_{p,r}(\widetilde{\Omega}) \leq B_{q,r}(\Omega) \cdot \|E_{\Omega}\|.
$$

By the Min-Max Principle $\mu_{p,r}(\widetilde{\Omega})^{-\frac{1}{p}}=B_{p,r}(\widetilde{\Omega})$. Thus, we finally have
$$
\mu_{p,r}(\widetilde{\Omega}) \geq \frac{1}{\|E_{\Omega}\|^p} \cdot \left(\mu_{q,r}(\Omega)\right)^{\frac{p}{q}}.
$$
\end{proof}

By using Theorem~\ref{thm:mon} we have the following lower estimates of non-linear Neumann eigenvalues.

\begin{thm} 
\label{thm:est}  
Let $\widetilde{\Omega}_{\widetilde{\gamma}} \subset \mathbb{B}^n$ be an outward $\widetilde{\gamma}$-cuspidal domain, $1<\widetilde{\gamma}<\infty$. Set $\gamma = \widetilde{\gamma} - 1$. Suppose that $\frac{n+(n-1)\gamma}{n} < p < n+(n-1)\gamma$, $1 \le q < \frac{n p}{n + (n-1)\gamma}$, $r<\frac{nq}{n-q}$.
Then the first non-trivial Neumann $(p,r)$-eigenvalue
\begin{multline*}
\mu_{p,r}(\widetilde{\Omega}_{\widetilde{\gamma}}) \\
\geq {\left(|\mathbb{B}^n|^{1-\alpha} + |\mathbb{S}^{n-2}|^{\,1-\alpha} \,C_{\gamma}^{\alpha} \left[ \frac{1-\alpha}{\,n(1-\alpha) - \alpha (n-1)\gamma\,} \right]^{1-\alpha}\right)^{-\frac{1}{\alpha}}}\left(\mu_{q,r}(\mathbb B^n)\right)^{\frac{1}{\alpha}},\\ 0<\alpha=\frac{q}{p}<1.
\end{multline*}
where 
$$
C_{\gamma}=\pi  \left(1+(\gamma+1)^2\frac{\pi^4}{(\pi-1)^4}+\frac{1}{(\pi-1)^2}\right)^{\frac{p}{2}},\,\,\gamma=\widetilde{\gamma}-1.
$$
\end{thm}

\vskip 0.2cm
\noindent

\section{Declarations}

\begin{itemize}

\item \textbf{Funding}: Not applicable.

\item \textbf{Conflict of interest/Competing interests}: The authors declare that there is no conflict of interest.

\item \textbf{Ethics approval}: Not applicable.

\item \textbf{Availability of data and materials}: Data sharing not applicable to this article as no datasets were generated or analysed during the current study.

\item \textbf{Authors' contributions}: The authors contributed equally to this work. 

\end{itemize}

\vskip 2cm

\noindent
Vladimir Gol'dshtein  \,  \hskip 3.2cm Alexander Ukhlov

\noindent
Department of Mathematics   \hskip 2.25cm Department of Mathematics

\noindent
Ben-Gurion University of the Negev  \hskip 1.05cm Ben-Gurion University of the Negev

\noindent
P.O.Box 653, Beer Sheva, 84105, Israel  \hskip 0.7cm P.O.Box 653, Beer Sheva, 84105, Israel

\noindent
E-mail: vladimir@bgu.ac.il  \hskip 2.5cm E-mail: ukhlov@math.bgu.ac.il

\end{document}